\documentclass[a4paper,10pt,twoside]{amsart}
\usepackage{hyperref}
\usepackage{amsfonts}
\usepackage{amssymb}
\usepackage{amsmath}
\usepackage[cmtip,all]{xy}
\usepackage{lscape}
\usepackage[usenames]{color}

\newtheorem{proposition}{Proposition}[section]
  \newtheorem{lemma}[proposition]{Lemma}
  \newtheorem{theorem}[proposition]{Theorem}
  
\theoremstyle{definition}
  \newtheorem{definition}[proposition]{Definition}
  \newtheorem*{definition*}{Definition}
  \newtheorem{example}[proposition]{Example}
  \newtheorem*{example*}{Example}

\numberwithin{equation}{section}

\def\id{\mathsf{id}}
\def\Aut{\mathsf{Aut}}
\def\can{\mathsf{can}}
\def\bc{\begin{center}}
\def\ec{\end{center}}

\begin{document}

\title{Hopf-Galois extensions for monoidal Hom-Hopf algebras}

\author{Yuanyuan Chen and Liangyun Zhang}
\address{College of Science, Nanjing Agricultural University, Nanjing
210095, P.R. China.}
\email{zlyun@njau.edu.cn}


\thanks{This work is supported by the National Natural Science
Foundation of China (10871170), the Educational Minister Science
Technology Key Foundation of China (108154), the College Special
Research Doctoral Disciplines Point Fund of China (20100097110040),
the Fundamental Research Funds for the Central Universities
(KYZ201125)
Corresponding author: Liangyun Zhang
}

\begin{abstract}
We investigate the theory of Hopf-Galois extensions for monoidal Hom-Hopf algebras.
As the main result of this paper, we prove the Schneider's affineness theorems in the case
of monoidal Hom-Hopf algebras in terms of the theory of the total integral and Hom-Hopf
Galois extensions. In addition, we obtain the affineness criterion for
relative Hom-Hopf module associated with faithfully flat Hom-Hopf Galois extensions.
\end{abstract}

\date{\today}
\maketitle

\section*{Introduction}

The study of nonassociative algebras was originally motivated by certain
problems in physics and other branches of mathematics. Hom-type algebras
appeared first in physical contexts, in connection with twisted, discretized
or deformed derivatives and corresponding generalizations, discretizations
and deformations of vector fields and differential calculus.
The notion of Hom-Lie algebras was introduced by Hartwig, Larsson, and Silvestrov
in \cite{D. Larsson2005,D. Larsson2007,J. T. Hartwig2006} as part
of a study of deformations of Witt algebras and Virasoro algebras. In a Hom-Lie algebra, the
Jacobi identity is twisted by a linear map, called the Hom-Jacobi identity
$$[\alpha(x),[y,z]]+[\alpha(y),[z,x]]+[\alpha(z),[x,y]]=0,$$
where $\alpha$ is a Lie algebra-endomorphism.
Because of close relation to discrete and deformed vector fields and differential calculus,
Hom-Lie algebras are widely studied recently, see
\cite{ A.Makhlouf2007,J. Arnlind,D. Yau2008,D. Yau20092,F. Ammar2010,
J. Arnlind2010,Q. Q. Jin2008, Y. Y. Chen2010,Y. Sheng}.

Hom-associative algebras play the role of associative algebras in the
Hom-Lie setting. They were introduced by Makhlouf and Silvestrov in \cite{A. Makhlouf2008}.
Hom-associative algebras and their related structures have recently become rather popular,
due to the prospect of having a general framework in which one can produce many
types of natural deformations of algebras. Among them such structures as
Hom-coassociative coalgebras, Hom-Hopf algebras, Hom-alternative algebras,
Hom-Jordan algebras, Hom-Poisson algebras, Hom-Leibniz algebras,
infinitesimal Hom-bialgebras, Hom-power associative algebras, quasi-triangular
Hom-bialgebras (see \cite{A. Gohr2010, A. Makhlouf20102,
A.Makhlouf2007,D. Yau20091,D. Yau20094,S. Caenepeel2009,Y. Y. Chen2013}), and so on.

Makhlouf and Silvestrov investigate Hom-associative algebras and
Hom-coassociative coalgebras further in \cite{A. Makhlouf20101,A.Makhlouf2007}.
Here the associativity of algebras, and the
coassociativity of coalgebras were twisted by endomorphisms. Hom-bialgebras
are both Hom-associative algebras and Hom-coassociative coalgebras such that
the comultiplication and counit are morphisms of algebras. These objects are
slightly different from the ones studied in this paper, see
Section 1.

The theory of Hopf Galois extensions which roots from the Galois theory for
groups acting on commutative rings, plays an important role in
the theory of Hopf algebras. There are two important applications of Hopf Galois
extensions: the Kreimer-Takeuchi type theorem and Schneider's affineness theorems.

The main purpose of this paper is to study the theory of Hopf-Galois extensions
for monoidal Hom-Hopf algebras. This paper is organized as follows.
In Section 2, relative Hom-Hopf modules are introduced.
In Section 3, we prove the affineness criterion for relative Hom-Hopf module
associated with faithfully flat Hom-Hopf Galois extensions.
In Section 4, we consider the Schneider's affineness theorems in the case
of monoidal Hom-Hopf algebras in terms of the theory of the total integral and Hom-Hopf
Galois extensions.

Throughout the paper, $k$ will be a fixed field and all vector
spaces, tensor products and homomorphisms are over $k$. We use
Sweedler's notation for coalgebras and comodules:
for a coalgebra $C$, we write its comultiplication
$\Delta(c)=c_{1}\otimes c_{2}$, for any $c\in C$; for a right
$C$-comodule $M$, we denote its coaction by  $\rho:m\mapsto
m_{(0)}\otimes m_{(1)}$, for any $m\in M$, in which we omit the
summation symbols for convenience. The category of $k$-modules will be
denoted by $\mathcal {M}_k$.

\section{Preliminaries}

Let $\mathcal {M}_k=(\mathcal {M}_k,\otimes,k,a,l,r)$ be the category of $k$-modules.
There is a new monoidal category $\mathcal{H}(\mathcal {M}_{k})$.
The objects of $\mathcal{H}(\mathcal {M}_{k})$ are couples $(M,\mu)$,
where $M\in \mathcal {M}_k$ and $\mu\in \Aut_k(M)$.
The morphisms of $\mathcal{H}(\mathcal {M}_{k})$ are morphisms $f:(M,\mu)\rightarrow (N,\nu)$
in $\mathcal {M}_k$ such that $\nu f=f\mu$.
For any objects $(M,\mu), (N,\nu)\in\mathcal{H}(\mathcal {M}_{k})$,
the monoidal structure is given by
$$(M,\mu)\otimes(N,\nu)=(M\otimes N,\mu\otimes\nu) \ \
\textrm{and}  \ \
(k,\id).$$

Briefly speaking, all Hom-structures are objects in the
monoidal category $\widetilde{\mathcal{H}}(\mathcal{M}_{k})=(\mathcal{H}(\mathcal {M}_{k}),$
$\otimes,(k,\id),\widetilde{a},\widetilde{l},\widetilde{r})$ introduced
in \cite{S. Caenepeel2009}, where the associator
$\widetilde{a}$ is given by the formula
\begin{eqnarray}
\widetilde{a}_{M,N,L}=a_{M,N,L}((\mu\otimes \id)\otimes \varsigma^{-1})
=(\mu\otimes (\id\otimes \varsigma^{-1})) a_{M,N,L}, \label{eq0}\end{eqnarray}
for any objects $(M,\mu),(N,\nu),(L,\varsigma)\in \mathcal{H}(\mathcal {M}_{k})$,
and the unitors $\widetilde{l}$ and $\widetilde{r}$ are
$$\widetilde{l}_M=\mu l_M=l_M(\id\otimes \mu); \ \ \widetilde{r}_M=\mu r_M=r_M(\mu \otimes \id).$$
The category $\widetilde{\mathcal{H}}(\mathcal{M}_{k})$ is called the
Hom-category associated to monoidal category $\mathcal {M}_k$,
where a $k$-submodule $N\subseteq M$ is called a subobject
of $(M,\mu)$ if $\mu$ restricts to an automorphism of $N$, that is,
$(N,\mu|_N)\in \widetilde{\mathcal{H}}(\mathcal{M}_{k})$.
Since the category $\mathcal {M}_{k}$ has left duality,
then so is the category $\widetilde{\mathcal{H}}(\mathcal{M}_{k})$.
Now let $M^*$ be the left dual of $M\in\mathcal {M}_{k}$,
and $b_M: k\rightarrow M\otimes M^*, d_M: M^*\otimes M\rightarrow k$
be the coevaluation and evaluation maps.
Then the left dual of $(M,\mu)\in\widetilde{\mathcal{H}}(\mathcal{M}_{k})$ is $(M^*,(\mu^*)^{-1})$,
and the coevaluation and evaluation maps are given by the formulas
$$ \widetilde{b}_M=(\mu\otimes\mu^*)^{-1} b_M; \ \ \widetilde{d}_M=d_M(\mu^*\otimes\mu).$$

In the following, we recall from \cite{S. Caenepeel2009} some information about Hom-structures.

\begin{definition}
A {\it unital monoidal Hom-associative algebra} is a vector space $A$
together with an element $1_{A}\in A$ and linear maps
$$m:A\otimes A \rightarrow A;a\otimes b\mapsto ab ,\hspace{1em} \alpha\in
Aut(A)$$  such that
\begin{eqnarray}
\alpha (a)(bc)=(ab)\alpha (c),\ \  a1_{A}=1_{A}a=\alpha (a), \label{eq1.2}
\end{eqnarray}
\begin{eqnarray}
\alpha (ab)=\alpha (a) \alpha (b), \ \ \alpha (1_{A})=1_{A}, \label{eq1.3}
\end{eqnarray}
for all $a,b,c \in A$.
\end{definition}

Note that the first part of $\eqref{eq1.2}$ can
be rewritten as $a(b\alpha^{-1}(c))=(\alpha^{-1}(a)b)c$.
In the language of Hopf algebras, $m$ is called {\it the
Hom-multiplication}, $\alpha$ is {\it the twisting automorphism} and
$1_{A}$ is {\it the unit}. Henceforth in this paper,
unless otherwise become necessary,
the terminology in Definition 1.1 will be slightly abused
for simplicity sake by making a convention to drop words unital and
Hom-associative. And we denote the monoidal Hom-algebra
by $(A,\alpha)$.

The definition of monoidal Hom-algebras is different from those defined in
\cite{A. Makhlouf20101,A.Makhlouf2007} in the following sense. The same twisted
associativity condition (\ref{eq1.2}) holds in both cases. However,
the unitality condition in their definition is the usual untwisted one:
$a1_A=1_Aa=a$, for any $a\in A$, and the twisting map
$\alpha$ does not need to be monoidal (that is, (\ref{eq1.3})
is not required).

Let $(A,\alpha)$ and $(A',\alpha')$ be two
monoidal Hom-algebras. A {\it Hom-algebra map}
$f:(A,\alpha)\rightarrow(A',\alpha')$ is a linear map such that
$f\alpha=\alpha' f, f(ab)=f(a)f(b)$ and $f(1_A)=1_{A'}$.

\begin{definition}
A {\it counital monoidal Hom-coassociative coalgebra} is an object $(C,\gamma)$
in the category $\widetilde{\mathcal{H}}(\mathcal {M}_{k})$ together
with linear maps $\Delta : C\rightarrow C\otimes C,
\Delta(c)=c_{1}\otimes c_{2}$ and $\varepsilon: C\rightarrow k$ such that
\begin{eqnarray}
\gamma^{-1}(c_{1})\otimes\Delta(c_{2})=\Delta(c_{1})\otimes
\gamma^{-1}(c_{2}), \ \ c_{1}\varepsilon(c_{2})=\gamma^{-1}(c)=\varepsilon(c_{1})c_{2}, \label{eq1.4}
\end{eqnarray}
\begin{eqnarray}
\Delta(\gamma(c))=\gamma(c_{1})\otimes\gamma(c_{2}), \ \ \varepsilon(\gamma(c))=\varepsilon(c), \label{eq1.5}
\end{eqnarray}
for all $c\in C.$
\end{definition}

Note that the first part of $\eqref{eq1.4}$ is equivalent to $c_{1}\otimes c_{21}\otimes
\gamma(c_{22})=\gamma(c_{11})\otimes c_{12}\otimes c_{2}$.
Analogue to monoidal Hom-algebras, monoidal Hom-coalgebras will be short for
counital monoidal Hom-coassociative coalgebras
without any confusion. The definition of monoidal Hom-coalgebra here
is somewhat different from the counital Hom-coassociative coalgebra
in \cite{A. Makhlouf20101,A.Makhlouf2007}.
Their coassociativity condition is twisted by some endomorphism, not necessarily by
the inverse of an automorphism, and the Hom-comultiplication is not comultiplicative.
The superiority of our definition is that these objects possess duality.
Then more results in Hopf algebras can be expanded to the monoidal-Hom case.

Let $(C,\gamma)$ and $(C',\gamma')$ be two monoidal Hom-coalgebras.
A {\it Hom-coalgebra map}
$f:(C,\gamma)\rightarrow(C',\gamma')$ is a linear map such that
$f\gamma=\gamma' f, \Delta f=(f\otimes f)\Delta$ and $\varepsilon f =\varepsilon$.

\begin{definition}
A {\it monoidal Hom-bialgebra}
$H=(H,\alpha,m,\eta,\Delta,\varepsilon)$ is a bialgebra in
the category $\widetilde{\mathcal{H}}(\mathcal{M}_{k})$. This
means that $(H,\alpha,m,\eta)$ is a monoidal Hom-algebra and
$(H,\alpha,\Delta,\varepsilon)$ is a monoidal Hom-coalgebra such that
$\vartriangle$ and $\varepsilon$ are Hom-algebra maps, that is, for
any $h, g\in H$,
$$\Delta(hg)=\Delta(h)\Delta(g),\ \Delta(1_{H})=1_{H}\otimes 1_{H}, $$
$$\varepsilon(hg)=\varepsilon(h)\varepsilon(g),\
\varepsilon(1_{H})=1_k.
$$
\end{definition}

For any bialgebra $(H,m,\eta,\Delta,\varepsilon)$, and any bialgebra endomorphism
$\alpha$ of $H$, the authors in \cite{A. Makhlouf20101} obtain that
$(H,\alpha,\alpha m,\eta,\Delta\alpha,\varepsilon)$ is a Hom-bialgebra in their terms.
In our case, there is a monoidal Hom-bialgebra $(H,\alpha,\alpha
m,\eta,\Delta\alpha^{-1},\varepsilon)$, provided that
$\alpha:H\rightarrow H$ is a bialgebra automorphism.

\begin{definition}
A monoidal Hom-bialgebra$(H,\alpha)$ is called a
{\it monoidal Hom-Hopf algebra} if there exists a morphism (called
antipode) $S : H\rightarrow H$ in $\widetilde{\mathcal{H}}(\mathcal{M}_{k})$
(i.e. $S\alpha=\alpha S$), such that for any $h\in H$,
\begin{eqnarray}
S(h_1)h_2=\varepsilon(h)1_H=h_1S(h_2). \label{eq1.6}
\end{eqnarray}
\end{definition}

In fact, a monoidal Hom-Hopf algebra is a Hopf algebra in the category
$\widetilde{\mathcal{H}}(\mathcal{M}_{k})$. Further, the antipodes of
monoidal Hom-Hopf algebras have similar properties of those of Hopf
algebras such as they are morphisms of Hom-anti-(co)algebras. Since $\alpha$
is bijective and commutes with antipode $S$, thus $S\alpha^{-1}=\alpha^{-1} S$.
For a finite-dimensional monoidal Hom-Hopf algebra $(H,\alpha,m,\eta,\Delta,\varepsilon,S)$,
the dual $(H^*,(\alpha^*)^{-1})$ is also a monoidal Hom-Hopf algebra with structures:
for all $g,h\in H,h^*,g^*\in H^*$,
$$<h^*g^*,h>=<h^*,h_1><g^*,h_2>, \ \ 1_{H^*}=\varepsilon;$$
$$<\Delta(h^*),g\otimes h>=<h^*,gh>, \ \ \varepsilon_{H^{*}}=\eta;$$
$$(\alpha^*)^{-1}(h^*)=h^*\alpha^{-1},  \ \ S^*(h^*)=h^* S^{-1}.$$

Then we recall the actions and  coactions over monoidal Hom-algebras and
monoidal Hom-coalgebras respectively.

\begin{definition}
Let $(A,\alpha)$ be a monoidal Hom-algebra. A {\it right $(A,\alpha)$-Hom-module}
consists of $(M,\mu)$ in $\widetilde{\mathcal{H}}(\mathcal{M}_{k})$
together with a morphism $\psi: M\otimes A\rightarrow M,
\psi(m\otimes a)=m\cdot a$ such that
$$(m\cdot a)\cdot\alpha(b)=\mu(m)\cdot(ab), \ \ m\cdot1_{A}=\mu(m),$$
$$\mu(m\cdot a)=\mu(m)\cdot\alpha(a),$$
for all $a,b\in A$ and $m\in M$.
\end{definition}

Similarly, we can define left $(A,\alpha)$-Hom-modules.
Monoidal Hom-algebra $(A,\alpha)$ can be considered
as a Hom-module on itself by the Hom-multiplication. Let
$(M,\mu),(N,\nu)$ be two left $(A,\alpha)$-Hom-modules. A morphism
$f:M\rightarrow N$ is called {\it left $(A,\alpha)$-linear}
(or {\it left $(A,\alpha)$-Hom-module map}) if $f(a\cdot m)=a\cdot
f(m),$ for any $a\in A, m\in M$, and $f\mu=\nu f$.
We denote the category of left $(A,\alpha)$-Hom-modules by
$\widetilde{\mathcal {H}}(_A\mathcal {M})$. If
$(M,\mu),(N,\nu)\in\widetilde{\mathcal {H}}(_H\mathcal {M})$, then
$(M\otimes N,\mu\otimes\nu)\in\widetilde{\mathcal {H}}(_H\mathcal
{M})$ via the left $(H,\alpha)$-action
\begin{eqnarray}
h\cdot(m\otimes n)=h_1\cdot m\otimes h_2\cdot n,\label{eq1.7}
\end{eqnarray}
where $(H,\alpha)$ is a monoidal Hom-bialgebra.

\begin{definition}
Let $(C,\gamma)$ be a monoidal Hom-coalgebra.
A {\it right $(C,\gamma)$-Hom-comodule} is an object
$(M,\mu)$ in $\widetilde{\mathcal{H}}(\mathcal{M}_{k})$ together with a
$k$-linear map $\rho_{M}:M\rightarrow M\otimes C,$
$\rho_{M}(m)=m_{(0)}\otimes m_{(1)}$ such that
$$
\mu^{-1}(m_{(0)})\otimes\Delta_{C}(m_{(1)})=m_{(0)(0)}\otimes(m_{(0)(1)}\otimes\gamma^{-1}(m_{(1)})),
\ \ \ m_{(0)}\varepsilon(m_{(1)})=\mu^{-1}(m),$$
$$\rho_{M}(\mu(m))=\mu(m_{(0)})\otimes\gamma(m_{(1)}),$$
for all $m\in M.$
\end{definition}

$(C,\gamma)$ is a Hom-comodule on itself via the
Hom-comultiplication. Let $(M,\mu),(N,\nu)$ be two right
$(C,\gamma)$-Hom-comodules. A morphism $g:M\rightarrow N$ is
called {\it right $(C,\gamma)$-colinear} (or {\it right
$(C,\gamma)$-Hom-comodule map}) if $g\mu=\nu g$ and $g(m_{(0)})\otimes
m_{(1)}=g(m)_{(0)}\otimes g(m)_{(1)}$, for any $m\in M$.
The category of right $(C,\gamma)$-Hom-comodules is denoted
by $\widetilde{\mathcal {H}}(\mathcal {M}^C)$. And we'll also denote
the set of morphisms in $\widetilde{\mathcal {H}}(\mathcal {M}^H)$ from $M$
to $N$ by $\widetilde{\mathcal {H}}(Com_H(M,N))$.
If $(M,\mu),(N,\nu)\in\widetilde{\mathcal{H}}(\mathcal {M}^H)$, then
$(M\otimes N,\mu\otimes\nu)\in\widetilde{\mathcal{H}}(\mathcal
{M}^H)$ with the Hom-comodule structure
\begin{eqnarray}
\rho(m\otimes n)=m_{(0)}\otimes n_{(0)}\otimes m_{(1)}n_{(1)}. \label{eq1.8}
\end{eqnarray}

In the following, we introduce the invariants and coinvariants on
Hom-modules and Hom-comodules respectively.

\begin{definition} Let $(H,\alpha)$ be a monoidal Hom-Hopf algebra.

(1) If $(M,\mu)$ is a left $(H,\alpha)$-Hom-module. The {\it invariant
of $(H,\alpha)$ on $(M,\mu)$} is the set
$$M^H=\{m\in M|h\cdot m=\varepsilon(h)\mu(m)\}.$$

(2) If $(N,\nu)$ is a right $(H,\alpha)$-Hom-comodule with the
comodule structure $\rho$. The {\it coinvariant of $(H,\alpha)$ on
$(N,\nu)$} is the set
$$N^{coH}=\{n\in N|\rho(n)=\nu^{-1}(n)\otimes 1_H\}.$$
\end{definition}

If $H$ is finite-dimensional, a right $(H,\alpha)$-Hom-comodule $(N,\nu)$ can
be considered as a left $(H^*,(\alpha^*)^{-1})$-Hom-module with the action
$h^*\cdot n=<h^*,n_{(1)}>\nu^2(n_{(0)})$, then we have
\begin{eqnarray}N^{coH}=\{n\in N|\rho(n)=\nu^{-1}(n)\otimes1_H\}
=\{n\in N|h^*\cdot n=<h^*,1_H>\nu(n)\}=N^{H^*}.\label{eq1.9}
\end{eqnarray}

\begin{definition} Let $(H,\alpha)$ be a monoidal Hom-Hopf algebra.
A {\it right $(H,\alpha)$-Hom-Hopf module}
$(M,\mu)$ is defined as a right $(H,\alpha)$-Hom-module and a
right $(H,\alpha)$-Hom-comodule as well, obeying the following
compatibility condition:
\begin{eqnarray}
\rho(m\cdot h)=m_{(0)}\cdot h_1\otimes m_{(1)}h_2,\label{eq1.10}
\end{eqnarray}
where $m\in M, h\in H$.
\end{definition}

Morphisms of right $(H,\alpha)$-Hom-Hopf modules are morphisms of both
right $(H,\alpha)$-linear and right $(H,\alpha)$-colinear.
We denote the category of right $(H,\alpha)$-Hom-Hopf modules
by $\mathcal {\widetilde{H}}(\mathcal{M}^H_H)$.

If $(M,\mu)$ is a right $(H,\alpha)$-Hom-Hopf module,
then so is $(M^{coH}\otimes H,\mu|_{M^{coH}}\otimes \alpha)$,
with the following action and coaction:
$$(m\otimes h)\cdot g=\mu(m)\otimes hg,$$
$$\rho(m\otimes h)=(\mu^{-1}(m)\otimes h_1)\otimes h_2,$$
where $m\in M, h,g\in H$.

\section{Relative Hom-Hopf modules}

In this section, we study relative Hom-Hopf modules and the adjoint functors
in terms to the category of relative Hom-Hopf modules.

\begin{definition}
Let $(H,\alpha)$ be a monoidal Hom-Hopf algebra.
A {\it right $(H,\alpha)$-Hom-comodule algebra} is a monoidal Hom-algebra
and a right $(H,\alpha)$-Hom-comodule $(A,\beta)$ with the coaction
$\rho_A: A\rightarrow A\otimes H$ such that $\rho_A$ is a morphism
of Hom-algebras, that is, for any $a,b\in A$,
\begin{eqnarray}\rho_A(ab)=\rho_A(a)\rho_A(b),   \label{eq2.1.1} \end{eqnarray}
$$ \rho_A(1_A)=1_A\otimes 1_H, $$
$$\rho_A\beta=(\beta\otimes\alpha)\rho_A.$$
\end{definition}

We always acquiesce $(A,\beta)$ being a right $(H,\alpha)$-Hom-comodule algebra.

Let $(H,m_H,\eta,\Delta,\varepsilon,S)$ be a Hopf algebra and $(A,m_A,\rho)$ a right $H$-comodule algebra.
If $\alpha:H\rightarrow H$ is a Hopf algebra automorphism, then there is a monoidal Hom-Hopf algebra
$H_\alpha=(H,m_\alpha=\alpha m_H,\eta,\Delta_\alpha=\Delta\alpha^{-1},\varepsilon,S,\alpha)$
by Proposition 1.14 in \cite{S. Caenepeel2009}. Let $\beta \in Aut(A)$ be an algebra automorphism such that
$\rho\beta=(\beta\otimes\alpha)\rho$, then it is easy to show that
$A_\beta=(A,m_\beta=\beta m_A,\rho_\beta=\rho\beta^{-1},\beta)$ is a
right $(H_\alpha,\alpha)$-Hom-comodule algebra by direct computation.
And the compatibility condition $\eqref{eq2.1.1}$ for $\rho_\beta$ and $m_\beta$ is just followed by
the compatibility $\rho(ab)=\rho(a)\rho(b)$ of comodule algebra $(A,m_A,\rho)$.

\begin{definition}
Let $(A,\beta,\rho_A)$ be a right
$(H,\alpha)$-Hom-comodule algebra. $(M,\mu)$ is called {\it a right $(H,A)$-Hom-Hopf
module} if $(M,\mu)$ is both in  $\mathcal{\widetilde{H}}(\mathcal {M}_A)$
and $\mathcal {\widetilde{H}}(\mathcal{M}^H)$
such that the following diagram commute:
$$\xymatrix{
M\otimes A \ar[r]^{\psi_M}\ar[d]_{\rho_M\otimes \rho_A} & M\ar[r]^{\rho_M}  & M\otimes H \\
(M\otimes H)\otimes(A\otimes H)\ar[d]_{\widetilde{a}}  &&  (M\otimes A)\otimes (H\otimes H)\ar[u]_{\psi_M\otimes m_H} \\
M\otimes(H\otimes(A\otimes H))\ar[d]_{\id\otimes\widetilde{a}^{-1}} &&  M\otimes (A\otimes (H\otimes H))\ar[u]_{\widetilde{a}^{-1}} \\
M\otimes((H\otimes A)\otimes H)\ar[rr]_{\id\otimes(\tau\otimes \id)} && M\otimes((A\otimes H)\otimes H)\ar[u]_{\id\otimes \widetilde{a}}
}$$
where $\psi_M$ is the right $(A,\beta)$-Hom-module action on $(M,\mu)$,
$\rho_M$ is the right $(H,\alpha)$-Hom-comodule structure of $(M,\mu)$,
$m_H$ is the multiplication of $H$, and $\tau$ is the flip map.
\end{definition}

The diagram is the compatibility condition for $(H,A)$-Hom-Hopf
module, which can be reexpressed as
\begin{eqnarray}\rho_M(m\cdot a)=m_{(0)}\cdot a_{(0)}\otimes m_{(1)}a_{(1)}, \label{eq2.2} \end{eqnarray}
for all $m\in M, a\in A$. We denote $\widetilde{\mathcal {H}}(M_A^H)$ as the category of right $(H,A)$-Hom-Hopf
modules and the right $(H,A)$-Hom-Hopf module morphisms, that is,
both right $(A,\beta)$-Hom-module maps and right $(H,\alpha)$-Hom-comodule morphisms between them.
Similarly, we can define $\widetilde{\mathcal {H}}(_A M^H)$ as the category of the left-right $(H,A)$-Hom-Hopf modules.

In fact, the right $(H,\alpha)$-Hom-comodule algebra
$(A,m_A,\rho_A,\beta)$ is itself a right $(H,A)$-Hom-Hopf-module
via the Hom-comodule structure $\rho_A$ and the Hom-multiplication
$m_A:A\otimes A\rightarrow A$, since the compatibility condition of
$(H,A)$-Hom Hopf modules is just the equality $\eqref{eq2.1.1}$.

\begin{example}
(1) We can induce a relative Hom-Hopf module from a
relative Hopf module $(M,\psi,\rho)$, which similar to induce a Hom-comodule algebra from
a comodule algebra. We just need to twist the action $\psi$ and the coaction $\rho$ into
$\psi_\mu=\mu\psi$ and $\rho_\mu=\rho\mu^{-1}$ respectively,
where $\mu:M\rightarrow M$ is an automorphism such that
$\mu\psi=\psi(\mu\otimes\beta)$ and $\rho\mu=(\mu\otimes\alpha)\rho$.

(2) Let $(A,\beta)$ be a right
$(H,\alpha)$-Hom-comodule algebra, $(M,\mu)$ be a right $(A,\beta)$-Hom-module, then
$(M\otimes H, \mu\otimes\alpha)$ is a right $(H,A)$-Hom-Hopf
module, with the right $(A,\beta)$-Hom-module structure
$\psi:(M\otimes H)\otimes A\rightarrow M\otimes H$; $(m\otimes h)\otimes a\mapsto
(m\otimes h)\cdot a=m\cdot a_{(0)}\otimes ha_{(1)}$
and the right $(H,\alpha)$-Hom-comodule structure
$\rho:M\otimes H\rightarrow (M\otimes H)\otimes H$;
$m\otimes h\mapsto (\mu^{-1}(m)\otimes h_1)\otimes \alpha(h_2)$.
Here we just check the compatibility condition \eqref{eq2.2}:
for any $m\in M, h\in H$ and $a\in A$,
$$\begin{array}{rllr}
(m\otimes h)_{(0)}\cdot a_{(0)}\otimes(m\otimes h)_{(1)}a_{(1)}
& =(\mu^{-1}(m)\otimes h_1)\cdot a_{(0)}\otimes\alpha(h_2)a_{(1)}
\\&  =(\mu^{-1}(m)\cdot a_{(0)(0)}\otimes h_1a_{(0)(1)})\otimes\alpha(h_2)a_{(1)}
\\&  =(\mu^{-1}(m)\cdot \beta^{-1}(a_{(0)})\otimes h_1a_{(1)1})\otimes\alpha(h_2)\alpha(a_{(1)2})
\\&  =(\mu^{-1}(m\cdot a_{(0)})\otimes h_1a_{(1)1})\otimes\alpha(h_2a_{(1)2})
\\&  =\rho(m\cdot a_{(0)}\otimes ha_{(1)})
=\rho ((m\otimes h)\cdot a).
\end{array}$$
In particular, $(A\otimes H,\beta\otimes\alpha)\in \widetilde{\mathcal {H}}(M_A^H)$.

\end{example}

If $(M,\mu)$ is a right $(H,\alpha)$-Hom-module and $(N,\nu)$ is a left $(H,\alpha)$-Hom-module,
{\it the tensor product over $(H,\alpha)$} of $(M,\mu)$ and $(N,\nu)$ in the category
$\widetilde{\mathcal{H}}(\mathcal{M}_{k})$ is defined as
\begin{eqnarray}\label{eq2.2.2}
(M\otimes_HN,\mu\otimes\nu)
=\{m\otimes n\in M\otimes N| m\cdot h\otimes\nu(n)
=\mu(m)\otimes h\cdot n\}.
\end{eqnarray}

And dually, if $(M,\mu)$ is a right $(H,\alpha)$-Hom-comodule and $(N,\nu)$ is a left $(H,\alpha)$-Hom-comodule,
{\it the co-tensor product space} $(M\Box_HN,\mu\otimes\nu)$ in the category $\widetilde{\mathcal{H}}(\mathcal{M}_{k})$
is defined as the following set:
\begin{eqnarray}\label{eq2.2.3}
\{m\otimes n\in M\otimes N| (m_{(0)}\otimes m_{(1)})\otimes\nu^{-1}(n)
=(\mu^{-1}(m)\otimes n_{(-1)})\otimes n_{(0)}\}.
\end{eqnarray}

Let $(A,\beta,\rho_A)$ be a right $(H,\alpha)$-Hom-comodule algebra.
We denote $B=A^{coH}$. If $(N,\nu)$ is a right $(H,A)$-Hom-Hopf modules,
then $(N^{coH},\nu|_{N^{coH}})$ is a right $(B,\beta|_B)$-Hom-submodule of
$(N,\nu)$. Obviously, $N\Box_{H} k\cong N^{coH}$, where $(k,\id)$ is a trivial
 $(H,\alpha)$-Hom-comodule.

For any right $(A,\beta)$-Hom-module $(M,\mu)$,
$(M\otimes_B A, \mu\otimes\alpha)$ is a right $(H,A)$-Hom-Hopf
module, with the action
$(m\otimes a)\otimes b\mapsto \mu(m)\otimes a\beta^{-1}(b)$,
and the coaction
$m\otimes a\mapsto (\mu^{-1}(m)\otimes a_{(0)})\otimes \alpha(a_{(1)})$.
This defines the induction functor
$F:\widetilde{\mathcal{H}}(\mathcal{M}_{B})\rightarrow \widetilde{\mathcal {H}}(M_A^H),
M\mapsto M\otimes_B A$. In fact, $F$ is a left adjoint to the functor of coinvariants
$G:\widetilde{\mathcal {H}}(M_A^H)\rightarrow \widetilde{\mathcal{H}}(\mathcal{M}_{B}),
N\mapsto N^{coH}$. (see the following result).

\begin{proposition}\label{prop:adjiont pair}
$(F,G)$ is a pair of adjoint functors with the unit
$$\eta_{(M,\mu)}: M\rightarrow (M\otimes_B A)^{coH}; \ \  m\mapsto \mu^{-1}(m)\otimes 1_A$$
and counit
$$\epsilon_{(N,\nu)}: N^{coH}\otimes_B A\rightarrow N; \ \ n\otimes a\mapsto n\cdot a,$$
where $(M,\mu)\in\widetilde{\mathcal{H}}(\mathcal{M}_{B}), (N,\nu)\in\widetilde{\mathcal {H}}(M_A^H)$.
\end{proposition}

\begin{proof}
Firstly, $\eta_{(M,\mu)}$ and $\epsilon_{(N,\nu)}$ are well-defined:
for any $m\in M$,
$\mu^{-1}(m)\otimes 1_A\in(M\otimes_B A)^{coH}$ is obvious,
and
$\epsilon_{(N,\nu)}(n\otimes ba)=n\cdot(ba)=\epsilon_{(N,\nu)}(\nu^{-1}(n)\cdot b\otimes\beta(a))$,
for any $n\in N$, and $a,b\in A$.
Then we need to check the triangular identity:
$$\epsilon_{F(M,\mu)}  F\eta_{(M,\mu)}(m\otimes a)=(\mu^{-1}(m)\otimes 1_A)\cdot a=m\otimes a,$$
$$G\epsilon_{(N,\nu)}  \eta_{G(N,\nu)}(n)=\nu^{-1}(n)\cdot 1_A=n.$$
\end{proof}

In the same way, the induction functor
$F:\widetilde{\mathcal{H}}(_B\mathcal{M})\rightarrow \widetilde{\mathcal {H}}(_A M^H),
M\mapsto A\otimes_B M$ is left adjoint to $N\mapsto N^{coH}$.

Similarly, for the left-right $(H,A)$-Hom-Hopf module category $\widetilde{\mathcal {H}}(M^H_A)$,
there is another pair of adjoint functors
$$F'=A\otimes_B-:\widetilde{\mathcal{H}}(_{B}\mathcal{M})\rightarrow \widetilde{\mathcal {H}}(_A M^H),$$
and
$$G'=(-)^{coH}:\widetilde{\mathcal {H}}(_A M^H)\rightarrow \widetilde{\mathcal{H}}(_{B}\mathcal{M}),$$
where $\widetilde{\mathcal{H}}(_{B}\mathcal{M})$ is the category of left $(B,\beta)$-Hom-modules.

\section{Hopf-Galois extensions}

In this section, we give some affineness theorems,
providing additional sufficient conditions for
$(F,G)$ or $(F',G')$ to be pairs of inverse equivalences.
We always assume that $(H,\alpha)$ is a monoidal Hom-Hopf
algebra with antipode $S$, and $(A,\beta)$ is a right
$(H,\alpha)$-Hom-comodule algebra.

Since $(A\otimes H,\beta\otimes\alpha)\in\widetilde{\mathcal {H}}(M^H_A)$
and $A$ is $k$-flat, we have
$$(A\otimes H)^{coH}\cong A\otimes H^{coH}\cong A\otimes k\cong A.$$
So the counit map in Proposition \ref{prop:adjiont pair}
is $\epsilon_{A\otimes H}:(A\otimes H)^{coH}\otimes_B A\rightarrow A\otimes H$
which can be translated to the following map:
$${\rm \can}: A\otimes_B A\rightarrow A\otimes H.$$
We find easily that
$${\rm \can} (a\otimes b)=(\beta^{-1}(a)\otimes 1_H)\cdot b=\beta^{-1}(a)b_{(0)}\otimes\alpha(b_{(1)}),$$
for all $a,b \in A$.

Similarly, $(A\otimes H,\beta\otimes\alpha)\in\widetilde{\mathcal {H}}(_A M^H)$,
and the corresponding adjunction map $\epsilon'_{A\otimes H}$ now defines another map
$${\rm \can'}: A\otimes_B A\rightarrow A\otimes H$$
given by
$${\rm \can'} (a\otimes b)=a\cdot(\beta^{-1}(b)\otimes 1_H)=a_{(0)}\beta^{-1}(b)\otimes\alpha(a_{(1)}).$$

\begin{proposition}
Let $(H,\alpha)$ be a monoidal Hom-Hopf algebras with a bijective antipode $S$,
$(A,\beta)$ a right $(H,\alpha)$-Hom-comodule algebra.
The map $f:A\otimes H\rightarrow A\otimes H$ given by
$$a\otimes h\mapsto \beta(a_{(0)})\otimes a_{(1)}S\alpha^{-1}(h)$$
is an isomorphism. Furthermore, ${\rm \can'}=f{\rm \can}$,
so ${\rm \can}$ is an isomorphism if and only if ${\rm \can'}$ is too.
\end{proposition}

\begin{proof}
For any $a\in A, h\in H$, it is easy to check that the inverse of $f$ is
$$f^{-1}(a\otimes h)=\beta(a_{(0)})\otimes S^{-1}\alpha^{-1}(h)a_{(1)},$$
by the Hom-coassociativity of Hom-comodule algebra $(A,\beta)$,
the Hom-associativity of $(H,\alpha)$ and the property of the antipode.
And for any $a,b\in A$,
$$\begin{array}{rllr}
f{\rm \can}(a\otimes b)
&=f(\beta^{-1}(a)b_{(0)}\otimes\alpha(b_{(1)}))
\\ &=\beta(\beta^{-1}(a)_{(0)}b_{(0)(0)})\otimes (\beta^{-1}(a)_{(1)}b_{(0)(1)})S(b_{(1)})
\\ &=a_{(0)}\beta(b_{(0)(0)})\otimes(\alpha^{-1}(a_{(1)})b_{(0)(1)})S(b_{(1)})
\\ &=a_{(0)}b_{(0)}\otimes(\alpha^{-1}(a_{(1)})b_{(1)1})S\alpha(b_{(1)2})
\\ &=a_{(0)}b_{(0)}\otimes a_{(1)}(b_{(1)1}S(b_{(1)2}))
\\ &=a_{(0)}b_{(0)}\otimes a_{(1)}\varepsilon(b_{(1)})1_H
\\ &=a_{(0)}\beta^{-1}(b)\otimes\alpha(a_{(1)})={\rm \can'}(a\otimes b)
\end{array}$$
\end{proof}

\begin{definition}
Consider a right $(H,\alpha)$-Hom-comodule algebra $(A,\beta)$ and its coinvariance $(B,\beta|_{B})$.
$(A,\beta)$ is called a Hopf-Galois extension of $(B,\beta|_{B})$ if ${\rm \can}$ or ${\rm \can'}$
is an isomorphism.
\end{definition}

If functors $(F,G)$ or $(F',G')$ is a pair of inverse equivalence of categories, then clearly
${\rm \can}$ and ${\rm \can'}$ are isomorphisms.

Now we consider the equivalent conditions for the pair $(F,G)$ to be inverse equivalence.
That is, $(A,\beta)$ is a faithfully flat Hopf-Galois extension.
This result dues to Doi and Takeuchi in \cite{Y. Doi}.

\begin{theorem} \label{theorem faithfull flat}
Let $(H,\alpha)$ be a monoidal Hom-Hopf algebras with a bijective antipode $S$,
$(A,\beta)$ a right $(H,\alpha)$-Hom-comodule algebra.
Then the following are equivalent:

1) $(A,\beta)$ is faithfully flat as a left $(B,\beta|_{B})$-Hom-module, and
$(A,\beta)$ is a Hopf-Galois extension of $(B,\beta|_{B})$;

2) $(F,G)$ is a pair of inverse equivalences between the categories
$\widetilde{\mathcal{H}}(_B\mathcal{M})$ and $\widetilde{\mathcal {H}}(M^H_A)$.
\end{theorem}

\begin{proof}
2) $\Rightarrow$ 1) We have already seen that $(A,\beta)$ is a Hopf-Galois extension
of $B$. Let $(M,\mu)\rightarrow (M',\mu')$ be an injective map of right
$(B,\beta)$-Hom-modules. The equivalence of categories
$\widetilde{\mathcal{H}}(_B\mathcal{M})$ and $\widetilde{\mathcal {H}}(M^H_A)$
implies that $M\otimes_B A\rightarrow M'\otimes_B A$ is monic in $\widetilde{\mathcal {H}}(M^H_A)$,
and of course monic in $\widetilde{\mathcal {H}}(M_A)$. Thus $(A,\beta)$ is left $(B,\beta|_{B})$-flat.
Faithfull flatness also follows by the equivalence of the categories
$\widetilde{\mathcal{H}}(_B\mathcal{M})$ and $\widetilde{\mathcal {H}}(M^H_A)$
in a similar way.

1) $\Rightarrow$ 2) For any $(N,\rho_N,\nu)\in\widetilde{\mathcal {H}}(M^H_A)$, we'll prove
that the counit $\epsilon_N$ is an isomorphism. If $(X,\iota)$ is a right
$(A,\beta)$-Hom-module, then the map ${\rm \can}_X$ is defined as the following composition
$$\xymatrix{
X\otimes_B A \ar[r]^-{\cong} &  X\otimes_A(A\otimes_B A) \ar[r]^-{X\otimes {\rm \can}}
&  X\otimes_A(A\otimes H) \ar[r]^-{\cong} & X\otimes H,
}$$
given by
$${\rm \can}_X(x\otimes a)=xa_{(0)}\otimes\alpha(a_{(1)}).$$
Since ${\rm \can}$ is an isomorphism, ${\rm \can}_X$ is also an isomorphism.
Now we have a commutative diagram as follows
$$
\xymatrix{
0 \ar[r]
&  N^{coH}\otimes_B A \ar[d]_-{\epsilon_N} \ar[r]
&  N\otimes_B A \ar[d]_-{{\rm \can}_N}
\ar@<2pt>[rr]^-{\rho_N\otimes_B A}
\ar@<-2pt>[rr]_-{(N\otimes \eta_H)\otimes A}
&&  (N\otimes H)\otimes_B A
\ar[d]^-{{\rm \can}_{N\otimes H}}
\\
0 \ar[r]
&  N  \ar[r]_-{\rho_N}
&  N\otimes H
\ar@<2pt>[rr]^-{\rho_N\otimes H}
\ar@<-2pt>[rr]_-{\widetilde{a}^{-1}(N\otimes\Delta_H)}
&&  (N\otimes H)\otimes H,
}
$$
The top row is exact, since $N^{coH}$ is the equalizer of $\rho_N$
and $N\otimes \eta_H$, and $A$ is flat as left $(B,\beta)$-Hom-modules.
Meanwhile, the equalizer of $\rho_N\otimes H$ and $\widetilde{a}^{-1}(N\otimes\Delta_H)$
is $N\Box_H H\cong N$. So the bottom row is also exact.
${\rm \can}_N$ and ${\rm \can}_{N\otimes H}$ are isomorphisms,
so $\epsilon_N$ is an isomorphism too by the short five lemma.

In addition, the unit $\eta_M: M\rightarrow (M\otimes_B A)^{coH}$ is also an isomorphism
by the following proof.
Define two maps $i_1, i_2: M\otimes_B A\rightarrow M\otimes_B(A\otimes_B A)$ as follows
$$i_1(m\otimes a)=m\otimes(1_A\otimes\beta^{-1}(a)), {\rm and} \ \ i_2(m\otimes a)=m\otimes(\beta^{-1}(a)\otimes 1_A),$$
for all $m\in M$ and $a\in A$. Then we have a commutative diagram
$$
\xymatrix{
0 \ar[r]
& M \ar[d]_-{\eta_M} \ar[r]^-{M\otimes \eta_A}
&  M\otimes_B A \ar@{=}[d]
\ar@<2pt>[rr]^-{i_1}
\ar@<-2pt>[rr]_-{i_2}
&&  M\otimes_B(A\otimes_B A)
\ar[d]^-{M\otimes{\rm \can}}
\\
0 \ar[r]
&  (M\otimes_B A)^{coH}  \ar[r]
& M\otimes_B A
\ar@<2pt>[rr]^-{M\otimes \rho_A}
\ar@<-2pt>[rr]_-{M\otimes(A\otimes\eta_H)}
&&  M\otimes_B(A\otimes H),
}
$$
by the definition of $i_1,i_2$ and $\can$, and the unitality of $(H,\alpha)$ and $(A,\beta)$.
The top row is exact because $A$ is faithfully flat as a left $(B,\beta)$-Hom-modules.
The bottom row is also exact by the definition of the coinvariants.
${\rm \can}$ is an isomorphism, so the adjunction unit $\eta_M$ is too,
again by the short five lemma.
\end{proof}

\section{total integrals}

In this section we consider the Schneider's affineness theorems under the assumption
that there exists a total integral.

\begin{definition}
Let $(A,\beta,\rho_A)$ be a right $(H,\alpha)$-Hom-comodule algebra.
The morphism $\varphi:(H,\alpha)\rightarrow (A,\beta)$ is called a total
integral for $(A,\beta)$ if $\varphi$ is a right $(H,\alpha)$-Hom-comodule map
such that $\varphi(1_H)=1_A$.
\end{definition}

We need some lemmas now for the main result.

\begin{lemma} (see \cite{Y. Y. Chen2014}) \label{lemma:relative}
Let $(A,\beta)$ be a right $(H,\alpha)$-Hom-comodule algebra.
Then the following are equivalent:

(1) there is a total integral,

(2) $(A,\beta)$ is an injective $(H,\alpha)$-Hom-comodule,

(3) all right $(H,A)$-Hom-Hopf module are injective as $(H,\alpha)$-Hom-comodule,

(4) there is a right $(H,\alpha)$-colinear map $\varphi: (H,\alpha)\rightarrow (A,\beta)$
with an invertible element $\varphi(1_H)$ in $A$.
\end{lemma}

Let $(M,\mu)\in\widetilde{\mathcal {H}}(^H M)$, then $(M,\mu)\in\widetilde{\mathcal {H}}(M^H)$
by $m\mapsto m_{(0)}\otimes S(m_{(-1)})$, for any $m\in M$. Applying the induction functor
$A\otimes-:\widetilde{\mathcal {H}}(M^H)\rightarrow \widetilde{\mathcal {H}}(_AM^H)$.
We find that $A\otimes M\in \widetilde{\mathcal {H}}(_AM^H)$, with structure
$$\psi^l: a\otimes(b\otimes m)\mapsto(\beta^{-1}(a)b)\otimes\mu(m),$$
$$\rho^r: a\otimes m\mapsto (a_{(0)}\otimes m_{(0)})\otimes a_{(1)}S(m_{(-1)}).$$

\begin{lemma} \label{lemma:coH}
With notations as above, we have
$$(A\otimes H)^{coH}=A\Box_HM,$$
for any $M\in\widetilde{\mathcal {H}}(^HM)$.
\end{lemma}

\begin{proof}
It is easy to check that $A\otimes M\in \widetilde{\mathcal {H}}(_AM^H)$ with the above structure.
Setting $a\otimes m\in(A\otimes M)^{coH}$, then we have
$$\rho^r(a\otimes m)=(a_{(0)}\otimes m_{(0)})\otimes a_{(1)}S(m_{(-1)})=(\beta^{-1}(a)\otimes\mu^{-1}(m))\otimes 1_H.$$
Firstly, we apply $\rho_m$ on the section factor, then
$$(a_{(0)}\otimes (m_{(0)(-1)}\otimes m_{(0)(0)}))\otimes a_{(1)}S(m_{(-1)})
=(\beta^{-1}(a)\otimes(\alpha^{-1}(m_{(-1)})\otimes\mu^{-1}(m_{(0)})))\otimes1_H.$$
Using the Hom-coassociativity of $(M,\mu)$ and applying $\alpha^2$ to the second factor,
then multiplying it to the last factor, we obtain
$$(a_{(0)}\otimes \mu^{-1}(m_{(0)}))\otimes\alpha(a_{(1)})(S\alpha(m_{(-1)1})\alpha(m_{(-1)2}))
=(\beta^{-1}(a)\otimes\mu^{-1}(m_{(0)}))\otimes\alpha^2(m_{(-1)}),$$
by the associativity of $(H,\alpha)$.
Using the property of the antipode, we have
$$(a_{(0)}\otimes\mu^{-2}(m))\otimes\alpha^2(a_{(1)})=(\beta^{-1}(a)\otimes\mu^{-1}(m_{(0)}))\otimes\alpha^2(m_{(-1)}),$$
which is equivalent to
$$(a_{(0)}\otimes a_{(1)})\otimes\mu^{-1}(m)=(\beta^{-1}(a)\otimes m_{(-1)})\otimes m_{(0)}.$$
This means that $a\otimes m$ is also in $A\Box_H M$.

Conversely, if $a\otimes m\in A\Box_H M$, then we get
$$\begin{array}{rllr}
\rho^r(a\otimes m) &= (a_{(0)}\otimes m_{(0)})\otimes a_{(1)}S(m_{(-1)})
\\ &=(\beta^{-1}(a)\otimes\mu(m_{(0)(0)}))\otimes m_{(-1)}S\alpha(m_{(0)(-1)})
\\ &=(\beta^{-1}(a)\otimes m_{(0)})\otimes \alpha(m_{(-1)1})S\alpha(m_{(-1)2})
\\ &=(\beta^{-1}(a)\otimes m_{(0)})\otimes \varepsilon(m_{(-1)})1_H
\\ &=(\beta^{-1}(a)\otimes \mu^{-1}(m))\otimes 1_H,
\end{array}$$
so $a\otimes m\in (A\otimes M)^{coH}$. Thus $(A\otimes H)^{coH}=A\Box_HM$.
\end{proof}

\begin{lemma} \label{lemma:ip}
If $(N,\nu)\in\widetilde{\mathcal {H}}(M_A^H)$, then we have well-defined maps
$$i: N^{coH} \rightarrow A\Box_H N; \ \ n\mapsto 1_A\otimes\nu^{-1}(n),$$
and
$$p: A\Box_H N \rightarrow N^{coH}; \ \ a\otimes n\mapsto n\cdot a$$
such that $pi=N^{coH}$, where the left $(H,\alpha)$-Hom-comodule on
$(N,\nu)$ is given by $n\mapsto S(n_{(1)})\otimes n_{(0)}$.
\end{lemma}

\begin{proof}
Firstly, the counitality and Hom-coassociativity of $(N,\nu)$ imply that
$(N,\nu)$ is also a left $(H,\alpha)$-Hom-comodule via $n\mapsto S(n_{(1)})\otimes n_{(0)}$.

Next, $i$ is well-defined, since taking $n\in N^{coH}$,
$i(n)=1_A\otimes\nu^{-1}(n)\in A\Box_H N$ is obvious by the left action on $N$
and the definition of coinvariants.

Also, $p$ is well-defined. Taking $a\otimes n\in A\Box_H N$, then
$$(a_{(0)}\otimes a_{(1)})\otimes\nu^{-1}(n)=(\beta^{-1}(a)\otimes S(n_{(1)}))\otimes n_{(0)}.$$
Applying $\rho_N$ to the last fact and using the Hom-coassociativity of $(N,\nu)$,
we have
\begin{eqnarray}\label{eq1}
(a_{(0)}\otimes a_{(1)})\otimes(n_{(0)}\otimes n_{(1)})
=(\beta^{-1}(a)\otimes S \alpha(n_{(1)2}))\otimes (n_{(0)}\otimes\alpha(n_{(1)1})).
\end{eqnarray}
Hence,
$$\begin{array}{rllr}
\rho_N(n\cdot a) &= n_{(0)}\cdot a_{(0)}\otimes n_{(1)}a_{(1)}
\\ &\stackrel{\eqref{eq1}}=n_{(0)}\cdot\beta^{-1}(a)\otimes\alpha(n_{(1)1})S\alpha(n_{(1)2})
\\ &=n_{(0)}\cdot\beta^{-1}(a)\otimes \varepsilon(n_{(1)})1_H
\\ &=\nu^{-1}(n)\cdot\beta^{-1}(a)\otimes 1_H
\\ &=\nu^{-1}(n\cdot a)\otimes 1_H
\end{array}$$
That is, $p(a\otimes n)\in N^{coH}$, which is as required.

Last, $p i(n)=p(1_A\otimes\nu^{-1}(n))=\nu^{-1}(n)\cdot 1_A=n$,
for all $n\in N$.
\end{proof}

\begin{proposition}
Let $(A,\beta)$ be a right $(H,\alpha)$-Hom-comodule algebra.
Then the following are equivalent:

1) $(A,\beta)$ is right $(H,\alpha)$-coflat;

2) $G=(-)^{coH}:\widetilde{\mathcal {H}}(M_A^H)\rightarrow \widetilde{\mathcal {H}}(M_B)$
is an exact functor;

3) $G'=(-)^{coH}:\widetilde{\mathcal {H}}(_A M^H)\rightarrow \widetilde{\mathcal {H}}(_B M)$
is an exact functor;
\end{proposition}

\begin{proof}
1)$\Rightarrow$2).  It's clear that $G$ is left exact.
Assume that $f:(N,\nu)\rightarrow (N',\nu')$ is surjective in
$\widetilde{\mathcal {H}}(M_A^H)$. $A\Box_H f$ is surjective
because $(A,\beta)$ is right $(H,\alpha)$-coflat. Since $f$
is a morphism in $\widetilde{\mathcal {H}}(M_k)$,
there is a commutative diagram
$$
\xymatrix{
A\Box_H N \ar[r]^-{A\Box_H f}\ar@<2pt>[d]^-{p}  &  A\Box_H N' \ar@<2pt>[d]^-{p}
\\ N^{coH} \ar[r]_{f} \ar@<2pt>[u]^-{i} & N'^{coH},\ar@<2pt>[u]^-{i}
}
$$
which implies that $f: N^{coH}\rightarrow N'^{coH}$ is surjective,
where $p,i$ are the maps defined in Lemma \ref{lemma:ip}.

3)$\Rightarrow$ 1). From Lemma \ref{lemma:coH}, we have known that
$(A\otimes M)^{coH}=A\Box_H M$. Then $A\Box_H (-)$ is the composition
of the functors
$$
\xymatrix{
\widetilde{\mathcal {H}}(^H M) \ar[r]^-{A\otimes(-)}
& \widetilde{\mathcal {H}}(_A M^H) \ar[r]^-{G'}
&  \widetilde{\mathcal {H}}(_B M).
}$$
$A\otimes (-)$ is exact since $(A,\beta)$ is $k$-flat, and $G'$ is also exact by the assumption.
It follows that $A\Box_H (-)$ is exact. Hence $(A,\beta)$ is right $(H,\alpha)$-coflat.

1)$\Rightarrow$ 3). We can apply ``1)$\Rightarrow$ 2)" to $A^{op}$ as $H^{op}$-Hom-comodule
algebra. Therefore,
$$N\rightarrow N^{coH}, N\in \widetilde{\mathcal {H}}(_A M^H)=\widetilde{\mathcal {H}}(M^{H^{op}}_{A^{op}})$$
is exact.

2)$\Rightarrow$ 1). This is done in a similar way: applying ``3)$\Rightarrow$ 1)" to $A^{op}$.
\end{proof}

\begin{lemma} \label{lemma:unit}
Assume that $(A,\beta)$ is a right $(H,\alpha)$-Hom-comodule algebra, and that
$\varphi:(H,\alpha)\rightarrow (A,\beta)$ is a total integral. For any $(M,\mu)\in\widetilde{\mathcal {H}}(M_B)$,
the adjunction unit $\eta_M: M\rightarrow (M\otimes_B A)^{coH}$ is an isomorphism.
\end{lemma}
\begin{proof}
Define a map in $\widetilde{\mathcal {H}}(M_k)$:
$$t:(A,\beta)\rightarrow (B,\beta|_B); a\mapsto a_{(0)}\varphi(S(a_{(1)})).$$
It's well-defined since
$$\begin{array}{rllr}
\rho_A(t(a)) &= a_{(0)(0)}(\varphi S(a_{(1)}))_{(0)}\otimes a_{(0)(1)}(\varphi S(a_{(1)}))_{(1)}
\\&=a_{(0)(0)}\varphi S(a_{(1)2})\otimes a_{(0)(1)}\varphi S(a_{(1)1})
\\&=\beta^{-1}(a_{(0)})\varphi S \alpha(a_{(1)22})\otimes a_{(1)1}S \alpha(a_{(1)21})
\\&=\beta^{-1}(a_{(0)})\varphi S(a_{(1)2})\otimes \alpha(a_{(1)11})S \alpha(a_{(1)12})
\\&=\beta^{-1}(a_{(0)})\varphi S(a_{(1)2})\otimes \varepsilon(a_{(1)1}) 1_H
\\&=\beta^{-1}(a_{(0)})\varphi S \alpha^{-1}(a_{(1)})\otimes 1_H
\\&=\beta^{-1}(a_{(0)}\varphi S(a_{(1)}))\otimes 1_H=\beta^{-1}(t(a))\otimes 1_H.
\end{array}$$
That is, $t(a)\in A^{coH}=B$.

Now define
$$\phi_M:(M\otimes_B A)^{coH}\rightarrow M, m\otimes a\mapsto m\cdot t(a).$$
In fact, $\phi_M$ is the inverse of $\eta_M$ by the following computation:
$$\begin{array}{rllr}
\phi_M\eta_M(m) &=\phi_M(\mu^{-1}(m)\otimes 1_A) =\mu^{-1}(m)\cdot t(1_A)
\\&=\mu^{-1}(m)\cdot(1_A\varphi S(1_H))=\mu^{-1}(m)\cdot 1_A=m,
\end{array}$$
and
$$\begin{array}{rllr}
\eta_M\phi_M(m\otimes a) & =\eta_M(m\cdot t(a))=\mu^{-1}(m\cdot t(a))\otimes 1_A
\\ &=\mu^{-1}(m)\cdot\beta^{-1} t(a)\otimes 1_A=m\otimes \beta^{-1} t(a)1_A
\\ &=m\otimes t(a)=m\otimes a_{(0)}\varphi S(a_{(1)})
\\ &=m\otimes \beta^{-1}(a)\varphi S(1_H)=m\otimes a,
\end{array}$$
where we use the facts $\beta^{-1}t(a)\in B$ in the fourth step,
and $m\otimes a\in(M\otimes_B A)^{coH}$ in the last step but one.
\end{proof}

\begin{theorem} \label{theorem total integral}
Let $(H,\alpha)$ be a monoidal Hom-Hopf algebra with a bijective antipode $S$,
$(A,\beta)$ a right $(H,\alpha)$-Hom-comodule algebra.
If ${\rm \can}$ is surjective, and if there exists a total integral
$\varphi:(H,\alpha)\rightarrow (A,\beta)$, then the adjoint pair $(F,G)$
between $\widetilde{\mathcal {H}}(M_A^H)$ and $\widetilde{\mathcal {H}}(M_B)$
is a pair of inverse equivalences.
\end{theorem}

\begin{proof}
We have shown that the unit of the adjunction is an isomorphism in Lemma \ref{lemma:unit}.
We only need to show the counit $\epsilon_N$ is an isomorphism too, for all
$N\in \widetilde{\mathcal {H}}(M_A^H)$.

Firstly, we prove this in the case that $N=V\otimes A$, where $(V,\nu)$ is an arbitrary
object in $\widetilde{\mathcal {H}}(M_k)$, and the $(H,A)$-Hom-Hopf module structure
on $(N,\nu\otimes\beta)$ is induced by the structure on $(A,\beta)$, that is,
$$(v\otimes a)\cdot b=\nu(v)\otimes a\beta^{-1}(b), \ \
\rho_{V\otimes A}(v\otimes a)=(\nu^{-1}(v)\otimes a_{(0)})\otimes\alpha(a_{(1)}),$$
for any $v\in V, a,b \in A$. By Lemma \ref{lemma:unit}, we have
$$(V\otimes A)^{coH}\cong (V\otimes(B\otimes_B A))^{coH}\cong ((V\otimes B)\otimes_B A)^{coH} \cong V\otimes B.$$
Then we have a commutative diagram
$$
\xymatrix{
(V\otimes B)\otimes_B A \ar[r]^-{\cong}\ar[d]_-{\cong}
& V\otimes(B\otimes_B A)\ar[d]^-{\cong}
\\ (V\otimes A)^{coH}\otimes_B A \ar[r]_-{\epsilon_{V\otimes A}}
& V\otimes A.
}$$
And we see that $\epsilon_{V\otimes A}$ is an isomorphism.

By Lemma \ref{lemma:relative}, we know the coaction $\rho_A: A\rightarrow A\otimes H$
on $A$ has a section $\lambda_A: A\otimes H\rightarrow A$. And $\lambda_A$ is a right
$(H,\alpha)$-Hom-comodule map with the explicit form as follows
$$\lambda_A(a\otimes h)=\beta(a_{(0)})\varphi(S(a_{(1)}\alpha^{-1}(h))),$$
for any $a\in A, h\in H$.

It is not difficult to check that $N\otimes(A\otimes H)\in\widetilde{\mathcal {H}}(M_A^H)$ with the structures
$$(n\otimes(a\otimes h))\cdot b=\nu^{-1}(n)\otimes(a\beta^{-1}(b_{(0)})\otimes h\alpha^{-1}(b_{(1)})),$$
$$\rho_{N\otimes(A\otimes H)}(n\otimes(a\otimes h))
=(\nu^{-1}(n)\otimes(\beta^{-1}(a)\otimes h_1))\otimes\alpha^2(h_2),$$
for all $a\in A, h\in H, n\in N$.
Define the map in $\widetilde{\mathcal {H}}(M_k)$:
$$
f: N\otimes(A\otimes H)\rightarrow N,
n\otimes(a\otimes h)\mapsto \nu(n_{(0)})\cdot\lambda_A(a\otimes S\alpha^{-1}(n_{(1)})\alpha^{-1}(h)).
$$
Firstly, $f$ is surjective, since for any $n\in N$,
$$\begin{array}{rllr}
f(n_{(0)}\otimes(1_A\otimes\alpha^{-1}(n_{(1)})))
&= \nu(n_{(0)(0)})\cdot\lambda_A(1_A\otimes S\alpha^{-1}(n_{(0)(1)})\alpha^{-2}(n_{(1)}))
\\ &=n_{(0)}\cdot\lambda_A(1_A\otimes S\alpha^{-1}(n_{(1)1})\alpha^{-1}(n_{(1)2}))
\\ &=n_{(0)}\cdot\lambda_A(1_A\otimes\varepsilon(n_{(1)})1_H)
=\nu^{-1}(n)\cdot 1_A=n.
\end{array}$$
Next, $f$ is right $(H,\alpha)$-colinear by the following computation:
for all $h\in H, a\in A$ and $n\in N$,
$$\begin{array}{rllr}
& \rho_Af(n\otimes(a\otimes h))
\\ & \ \ =\rho_A(\nu(n_{(0)})\cdot\lambda_A(a\otimes S\alpha^{-1}(n_{(1)})\alpha^{-1}(h)))
\\ & \ \ =\nu(n_{(0)(0)})\cdot\lambda_A(a\otimes S\alpha^{-1}(n_{(1)})\alpha^{-1}(h))_{(0)}
          \otimes \alpha(n_{(0)(1)})\lambda_A(a\otimes S\alpha^{-1}(n_{(1)})\alpha^{-1}(h))_{(1)}
\\ & \ \ =\nu(n_{(0)(0)})\cdot\lambda_A(\beta^{-1}(a)\otimes(S\alpha^{-1}(n_{(1)})\alpha^{-1}(h))_1)
                                  \otimes\alpha(n_{(0)(1)})\alpha((S\alpha^{-1}(n_{(1)})\alpha^{-1}(h))_2)
\\ & \ \ =\nu(n_{(0)(0)})\cdot\lambda_A(\beta^{-1}(a)\otimes S\alpha^{-1}(n_{(1)2})\alpha^{-1}(h_1))\otimes\alpha(n_{(0)(1)})(S(n_{(1)1})h_2)
\\ & \ \ =n_{(0)}\cdot\lambda_A(\beta^{-1}(a)\otimes S(n_{(1)22})\alpha^{-1}(h_1))\otimes\alpha(n_{(1)1})(S\alpha(n_{(1)21})h_2)
\\ & \ \ =n_{(0)}\cdot\lambda_A(\beta^{-1}(a)\otimes S\alpha^{-1}(n_{(1)2})\alpha^{-1}(h_1))\otimes\alpha^2(n_{(1)11})(S\alpha(n_{(1)12})h_2)
\\ & \ \ =n_{(0)}\cdot\lambda_A(\beta^{-1}(a)\otimes S\alpha^{-1}(n_{(1)2})\alpha^{-1}(h_1))\otimes(\alpha(n_{(1)11})S\alpha(n_{(1)12}))\alpha(h_2)
\\ & \ \ =n_{(0)}\cdot\lambda_A(\beta^{-1}(a)\otimes S\alpha^{-1}(n_{(1)2})\alpha^{-1}(h_1))\otimes\varepsilon(n_{(1)1})1_H\alpha(h_2)
\\ & \ \ =n_{(0)}\cdot\lambda_A(\beta^{-1}(a)\otimes S\alpha^{-2}(n_{(1)})\alpha^{-1}(h_1))\otimes\alpha^2(h_2)
\\ & \ \ =(f\otimes H)((\nu^{-1}(n)\otimes(\beta^{-1}(a)\otimes h_1))\otimes \alpha^2(h_2))
\\ & \ \ =(f\otimes H)\rho_{N\otimes(A\otimes H)}(n\otimes(a\otimes h)),
\end{array}$$
where the second step follows by the compatibility \eqref{eq2.2}, and the third step
holds since $\lambda_A$ is right $(H,\alpha)$-colinear and the coaction on
$(A\otimes H,\beta\otimes\alpha)$ is given by
$\rho(a\otimes h)=(\beta^{-1}(a)\otimes h_1)\otimes \alpha(h_2)$.
We conclude that $f$ a split epimorphism in $\widetilde{\mathcal {H}}(M^H)$ finally.

Since $H$ is projective as a $k$-module, $A\otimes H$ is projective as a
left $(A,\beta)$-Hom-module. The map ${\rm \can}: A\otimes A\rightarrow A\otimes H$
is a left $(A,\beta)$-linear epimorphism because
$$\begin{array}{rllr}
{\rm \can}(c\cdot (a\otimes b))= & {\rm \can}(\beta^{-1}(c)a\otimes\beta(b))
=(\beta^{-2}(c)\beta^{-1}(a))\beta(b_{(0)})\otimes\alpha^2(b_{(1)})
\\ &=\beta^{-1}(c)(\beta^{-1}(a)b_{(0)})\otimes\alpha^2(b_{(1)})
=c\cdot{\rm \can}(a\otimes b).
\end{array}$$
Thus ${\rm \can}$ has an $(A,\beta)$-linear splitting,
and a fortiori splitting in $\widetilde{\mathcal {H}}(M_k)$.

It is easy to check that $N\otimes(A\otimes A)\in \widetilde{\mathcal {H}}(M_A^H)$
with the following structures
$$(n\otimes(a\otimes a'))\cdot b=\nu(n)\otimes(\alpha(a)\otimes a'\beta^{-2}(b)),$$
and
$$\rho_{N\otimes(A\otimes A)}=(\nu^{-1}(n)\otimes(\beta^{-1}(a)\otimes b_{(0)}))\otimes\alpha^{2}(b_{(1)}).$$
Then
$$N\otimes{\rm \can}:N\otimes(A\otimes A)\rightarrow N\otimes(A\otimes H)$$
is a morphism in $\widetilde{\mathcal {H}}(M_A^H)$, which is surjective and
split in $\widetilde{\mathcal {H}}(M_k)$. Therefore,
$$g=f(N\otimes{\rm \can}): N\otimes(A\otimes A)\rightarrow N$$
is surjective and split in $\widetilde{\mathcal {H}}(M_k)$.

Set $N'=ker(g)$. Then there is an exact sequence
\begin{equation}\label{exact sequene}
\xymatrix{
0 \ar[r]
& N' \ar[r]
& N\otimes(A\otimes A) \ar[r]^-{g}
& N \ar[r]
& 0
}
\end{equation}
in $\widetilde{\mathcal {H}}(M_A^H)$ which is split as a sequence in $\widetilde{\mathcal {H}}(M_k)$.
\eqref{exact sequene} is also a split exact sequence of $(H,\alpha)$-Hom-comodule
by Lemma \ref{lemma:relative}.

Repeating the resolution with $N'$ instead of $N$, we obtain another
exact sequence in $\widetilde{\mathcal {H}}(M_A^H)$
\begin{equation}\label{exact1}
\xymatrix{
0 \ar[r]
& N'' \ar[r]
& N'\otimes(A\otimes A) \ar[r]^-{g'}
& N' \ar[r]
& 0
}
\end{equation}
which is split in $\widetilde{\mathcal {H}}(M^H)$.
Now denote $N_1=N\otimes(A\otimes A), N_2=N'\otimes(A\otimes A)$.
Combining the two sequences \eqref{exact sequene} and \eqref{exact1},
the following exact sequence
$$
\xymatrix{
N_2 \ar[r]^-{g'}
& N_1 \ar[r]^-{g}
& N \ar[r]
& 0
}
$$
is obtained in $\widetilde{\mathcal {H}}(M^H)$.
Since \eqref{exact sequene} and \eqref{exact1} are both split exact in
$\widetilde{\mathcal {H}}(M^H)$, they stay exact after we take
$(H,\alpha)$-coinvariants and combine them. Then we have an exact
sequence in $\widetilde{\mathcal {H}}(M_B)$
$$
\xymatrix{
N_2^{coH} \ar[r]
& N_1^{coH} \ar[r]
& N^{coH} \ar[r]
& 0
}
$$
The tensor functors are always right exact,
so finally we obtain an exact sequence
$$
\xymatrix{
N_2^{coH}\otimes_B A \ar[r]
& N_1^{coH}\otimes_B A \ar[r]
& N^{coH}\otimes_B A \ar[r]
& 0
}
$$
in $\widetilde{\mathcal {H}}(M_A^H)$.
Thus, there is a commutative diagram
$$
\xymatrix{
N_2^{coH}\otimes_B A \ar[r]\ar[d]_-{\epsilon_{N_2}}
& N_1^{coH}\otimes_B A \ar[r]\ar[d]_-{\epsilon_{N_1}}
& N^{coH}\otimes_B A \ar[r]\ar[d]_-{\epsilon_N}
& 0
\\
N_2 \ar[r]_-{g'}
& N_1 \ar[r]_-{g}
& N \ar[r]
& 0
}
$$
where both the bottom and the top lines are exact sequences in
$\widetilde{\mathcal {H}}(M_A^H)$. Since $N_1=N\otimes(A\otimes A)\cong(N\otimes A)\otimes A$
and $N_2=N'\otimes(A\otimes A)\cong(N'\otimes A)\otimes A$ are $(H,A)$-Hom-Hopf modules
of the form $V\otimes A$, where $V$ is an arbitrary object in $\widetilde{\mathcal {H}}(M_k)$,
we see that $\epsilon_{N_1}$ and $\epsilon_{N_2}$ are isomorphisms. Thus
$\epsilon_{N}$ is an isomorphism too.
\end{proof}

%
%
%
%

\begin{lemma}  \label{lemma:coflat}
Let $(H,\alpha)$ be a monoidal Hom-Hopf algebra. We assume that
there is an isomorphism $M\Box_H Q\cong\widetilde{\mathcal{H}}(Com_H(Q^*,M))$,
for any right $(H,\alpha)$-Hom-comodule $(M,\mu)$ and any finite-dimensional left
$(H,\alpha)$-Hom-comodule $(Q,\kappa)$. Then $(M,\mu)$ is right $(H,\alpha)$-coflat
if and only if it is an injective object in $\widetilde{\mathcal {H}}(M^H)$.
\end{lemma}

\begin{proof}
If $(M,\mu)$ is injective in $\widetilde{\mathcal {H}}(M^H)$, then there is an
$(H,\alpha)$-Hom-colinear map
$$\lambda_M:M\otimes H\rightarrow M$$ splitting $\rho_M$, that is,
$\lambda_M\rho_M=\id_M$. Let $f:(N,\nu)\rightarrow (W,\omega)$ be surjective in
$\widetilde{\mathcal {H}}(M^H)$ and take $m\otimes w\in M\Box_H W$. Since $f$
is surjective, we can find an element $n\in N$ such that $f(n)=w$. To show
$(M,\mu)$ is right $(H,\alpha)$-coflat, we only need to show $M\Box_H f:M\Box_H N\rightarrow M\Box_H W$
is surjective. In fact, $m\otimes w\in M\Box_H W$ implies that
$$m\otimes w=\lambda_M(m_{(0)}\otimes m_{(-1)})\otimes f(n)
=(M\otimes f)(\lambda_M(\mu^{-1}(m)\otimes n_{(-1)})\otimes\nu(n_{(0)})).$$
Using the fact that $\lambda_M$ is $(H,\alpha)$-colinear, we have
$$\begin{array}{rllr}
 & (\rho_M\otimes M)(\lambda_M(\mu^{-1}(m)\otimes n_{(-1)})\otimes\nu(n_{(0)}))
\\ &=(\lambda_M(\mu^{-2}(m)\otimes n_{(-1)1})\otimes\alpha(n_{(-1)2}))\otimes\nu(n_{(0)})
\\ &=(\lambda_M(\mu^{-2}(m)\otimes \alpha^{-1}(n_{(-1)}))\otimes\alpha(n_{(0)(-1)}))\otimes\nu^2(n_{(0)(0)})
\\ &=\widetilde{a}^{-1}(M\otimes\rho_N)(\lambda_M(\mu^{-1}(m)\otimes n_{(-1)})\otimes\nu(n_{(0)})).
\end{array}$$
So $\lambda_M(\mu^{-1}(m)\otimes n_{(-1)})\otimes\nu(n_{(0)})\in M\Box_H N$ and this shows
that $M\Box_H f$ is surjective.

Conversely, if $(N,\mu,\rho_N)$ is a finite dimensional right $(H,\alpha)$-Hom-comodule, then
the natural morphism
$$
\theta: H\otimes N^*\rightarrow Hom(N,H), \theta(h\otimes n^*)(n)=n^*(\nu(m))\alpha(h)
$$
is a linear isomorphism. $(N^*,(\nu^*)^{-1})$ is a left $(H,\alpha)$-Hom-comodule via the coaction
$$\rho_{N^*}: N^*\rightarrow H\otimes N^*, \rho_{N^*}(n^*)=\theta^{-1}((n^*\otimes H)\rho_N).$$
Then
$$M\Box_H N^*\cong \widetilde{\mathcal {H}}(Com_H(N^{**},M))\cong \widetilde{\mathcal {H}}(Com_H(N,M)),$$
by the assumption.
Since $(M,\mu)$ is coflat we obtain that $(M,\mu)$ is $(N,\nu)$-injective
which means that for any Hom-subcomodule $(N',\nu)$ of $(N,\nu)$, and for any
$f\in \widetilde{\mathcal {H}}(Com_H(N',M))$, there exists $g\in \widetilde{\mathcal {H}}(Com_H(N,M))$
such that $g|_{(N',\nu)}=f$.
Then  $(M,\mu)$ is also an injective object in $\widetilde{\mathcal {H}}(M^H)$, where the
proof is similar to the non-Hom-case in Theorem 2.4.17 in \cite{S. Dascalescu}.
\end{proof}

\begin{theorem}
Let $(H,\alpha)$ be a monoidal Hom-Hopf algebra with a bijective antipode,
$(A,\beta)$ a right $(H,\alpha)$-Hom-comodule algebra. We assume that
there is an isomorphism $M\Box_H Q\cong\widetilde{\mathcal{H}}(Com_H(Q^*,M))$,
for any right $(H,\alpha)$-Hom-comodule $(M,\mu)$ and any finite-dimensional left
$(H,\alpha)$-Hom-comodule $(Q,\kappa)$. Then the following assertions
are equivalent.

1) There exits a total integral $\varphi:(H,\alpha)\rightarrow (A,\beta)$ and the map ${\rm \can}$ is surjective;

2) The functors $F$ and $G$ are a pair of inverse equivalence between the categories
$\widetilde{\mathcal {H}}(M_A^H)$ and $\widetilde{\mathcal {H}}(M_B)$;

3) The functors $F'$ and $G'$ are a pair of inverse equivalence between the categories
$\widetilde{\mathcal {H}}(_A M^H)$ and $\widetilde{\mathcal {H}}(_B M)$;

4) $A$ is a Hopf-Galois extension of $B$, and is faithfull flat as a left $(B,\beta)$-Hom-module;

5) $A$ is a Hopf-Galois extension of $B$, and is faithfull flat as a right $(B,\beta)$-Hom-module;
\end{theorem}

\begin{proof}
1)$\Rightarrow$ 2) follows by Theorem \ref{theorem total integral}, and 2)$\Leftrightarrow$ 4)
follows by Theorem \ref{theorem faithfull flat}. Now we only need to show 4)$\Rightarrow$ 1).
Suppose that $A$ is a Hopf-Galois extension of $B$, and is faithfull flat as a left $(B,\beta)$-Hom-module.
In order to show there is a total integral, we only need to show that $(A,\beta)$ is an injective object
in $\widetilde{\mathcal {H}}(M^H)$ by Lemma \ref{lemma:relative}. Equivalently, we have to show
that $(A,\beta)$ is right $(H,\alpha)$-coflat by Lemma \ref{lemma:coflat}.

For any $(V,\nu)\in\widetilde{\mathcal {H}}(^H M)$, $A\Box_HV$ is a right $(B,\beta|_B)$-Hom-module
via $(a\otimes v)\cdot b=a\beta^{-1}(b)\otimes\nu(v)$. Define a map:
$$\varpi: (A\Box_H V)\otimes_B A\rightarrow (A\otimes_B A)\Box_H V;
(a\otimes v)\otimes a'\mapsto (a\otimes\beta^{-1}(a'))\otimes\nu(v),$$
where the right $(H,\alpha)$-Hom-comodule structure on $A\otimes_B A$ is given by
$\rho_{A\otimes_B A}(a\otimes a')=(a_{(0)}\otimes\beta^{-1}(a'))\otimes\alpha(a_{(1)})$.
Since $A$ is flat as left $B$-Hom-modules, $\varpi$ is an isomorphism as left $B$-Hom-modules.
${\rm \can}$ is bijective, so ${\rm \can'}$ is an isomorphism. Then we have the following
sequence of left $B$-Hom-module isomorphisms:
$$
(A\Box_H V)\otimes_B A\cong (A\otimes_B A)\Box_H V \cong (A\otimes H)\Box_H V \cong A\otimes(H\Box_H V) \cong A\otimes V.
$$
For any exact sequence
\begin{equation}
\xymatrix
{ 0 \ar[r] &  U \ar[r]  & V \ar[r]   &   W   \ar[r]  &  0  }
\end{equation}
in $\widetilde{\mathcal {H}}(^H M)$, the sequence
\begin{equation}
\xymatrix
{0 \ar[r]  & A\otimes U \ar[r]  &  A\otimes V \ar[r]  &  A\otimes W  \ar[r]  &  0 }
\end{equation}
is also exact in $\widetilde{\mathcal {H}}(_k M)$ as $k$ is a field.
Hence, we have the following exact sequence
\begin{equation}
\xymatrix
{0 \ar[r]  &  (A\Box_H U)\otimes_B A  \ar[r]  &  (A\Box_H V)\otimes_B A \ar[r]  &  (A\Box_H W)\otimes_B A  \ar[r]  & 0 }.
\end{equation}
Since $A$ is faithfully flat as a left $B$-Hom-module, the following exact sequence
\begin{equation}
\xymatrix
{0 \ar[r]  &  A\Box_H U  \ar[r]  &  A\Box_H V  \ar[r]  &  A\Box_H W  \ar[r]  & 0 }
\end{equation}
is obtained at last, and this implies that $A$ is right $(H,\alpha)$-coflat.
\end{proof}

\end{document}